\documentclass[a4paper,11pt,fleqn]{article}
\usepackage[obeyspaces,hyphens,spaces]{url}
\usepackage[dvipsnames]{xcolor}
\usepackage{amsmath}
\usepackage{amssymb}
\usepackage{amsfonts}
\usepackage{mathtools}
\usepackage{euscript}
\usepackage[scr=rsfs]{mathalfa}
\usepackage{pifont}     
\usepackage{srcltx}     
\usepackage{charter}
\usepackage{etoolbox}
\allowdisplaybreaks 
\definecolor{labelkey}{HTML}{0455BF}
\definecolor{refkey}{rgb}{0,0.6,0.0}
\definecolor{dblue}{HTML}{044EAF}
\definecolor{dgreen}{HTML}{02724A}
\definecolor{myellow}{HTML}{D97904}
\definecolor{dred}{HTML}{D90404}
\usepackage{upref}
\usepackage{hyperref}
\hypersetup{%
pdftitle={A Decomposition Method for Solving Multicommodity
Network Equilibrium},
pdfauthor={Minh N. B\`ui},
colorlinks=true,linkcolor=dblue,citecolor=dgreen,urlcolor=dred}
\renewcommand\familydefault{\rmdefault}
\DeclareMathAlphabet{\mathrm}{OT1}{\familydefault}{m}{n}
\makeatletter
\def\operator@font{\mathgroup\symoperators\rm}
\makeatother
\DeclareFontFamily{U}{matha}{\hyphenchar\font45}
\DeclareFontShape{U}{matha}{m}{n}{
<-6> matha5 <6-7> matha6 <7-8> matha7
<8-9> matha8 <9-10> matha9
<10-12> matha10 <12-> matha12
}{}
\DeclareSymbolFont{matha}{U}{matha}{m}{n}
\DeclareFontFamily{U}{mathx}{\hyphenchar\font45}
\DeclareFontShape{U}{mathx}{m}{n}{
<-6> mathx5 <6-7> mathx6 <7-8> mathx7
<8-9> mathx8 <9-10> mathx9
<10-12> mathx10 <12-> mathx12
}{}
\DeclareSymbolFont{mathx}{U}{mathx}{m}{n}
\DeclareMathDelimiter{\vvvert}{0}{matha}{"7E}{mathx}{"17}%
\renewcommand{\leq}{\ensuremath{\leqslant}}
\renewcommand{\geq}{\ensuremath{\geqslant}}

\newcommand{\scal}[2]{\langle{{#1}\mid{#2}}\rangle}

\newcommand{\menge}[2]{\big\{{#1}~|~{#2}\big\}} 
\newcommand{\Menge}[2]{\left\{{#1}~\middle|~{#2}\right\}} 
\DeclareFontFamily{U}{BOONDOX-calo}{\skewchar\font=45}
\DeclareFontShape{U}{BOONDOX-calo}{m}{n}{
  <-> s*[1.05] BOONDOX-r-calo}{}
\DeclareFontShape{U}{BOONDOX-calo}{b}{n}{
  <-> s*[1.05] BOONDOX-b-calo}{}
\DeclareMathAlphabet{\mathcalboondox}{U}{BOONDOX-calo}{m}{n}
\SetMathAlphabet{\mathcalboondox}{bold}{U}{BOONDOX-calo}{b}{n}
\DeclareMathAlphabet{\mathbcalboondox}{U}{BOONDOX-calo}{b}{n}
\newcommand{\BN}{\mathcalboondox{N}}
\newcommand{\BA}{\mathcalboondox{A}}
\newcommand{\BCC}{\mathcalboondox{C}}

\newcommand{\XXX}{\ensuremath{\boldsymbol{\mathcal{X}}}}
\newcommand{\VVV}{\ensuremath{\boldsymbol{\mathcal{V}}}}
\newcommand{\HH}{\mathcal{H}}

\newcommand{\GG}{\mathcal{G}}

\newcommand{\RR}{\mathbb{R}}
\newcommand{\NN}{\mathbb{N}}
\newcommand{\fc}{\mathfrak{c}}

\newcommand{\pinf}{\ensuremath{{{+}\infty}}}

\newcommand{\RP}{\ensuremath{\left[0,{+}\infty\right[}}
\newcommand{\RPX}{\ensuremath{\left[0,{+}\infty\right]}}
\newcommand{\RX}{\ensuremath{\left]{-}\infty,{+}\infty\right]}}

\newcommand{\RPP}{\ensuremath{\left]0,{+}\infty\right[}}

\newcommand{\emp}{\varnothing}

\newcommand{\Id}{\mathrm{Id}}

\newcommand{\exi}{\ensuremath{\exists\,}}
\newcommand{\prox}{\mathrm{prox}}
\newcommand{\proj}{\mathrm{proj}}

\DeclareMathOperator{\card}{card}

\DeclareMathOperator{\dom}{dom}

\DeclareMathOperator{\gra}{gra}

\DeclareMathOperator{\dive}{div}
\newcommand*\mute{{\mkern 2mu\cdot\mkern 2mu}}

\def\abstract{\noindent{\bfseries Abstract}. \ignorespaces}

\usepackage{theorem}
\newtheorem{theorem}{Theorem}

\newtheorem{proposition}[theorem]{Proposition}

\theoremstyle{plain}{\theorembodyfont{\rmfamily}%
}

\theoremstyle{plain}{\theorembodyfont{\rmfamily}%
\newtheorem{example}[theorem]{Example}}
\theoremstyle{plain}{\theorembodyfont{\rmfamily}%
\newtheorem{remark}[theorem]{Remark}}
\theoremstyle{plain}{\theorembodyfont{\rmfamily}%
}
\theoremstyle{plain}{\theorembodyfont{\rmfamily}%
}
\theoremstyle{plain}{\theorembodyfont{\rmfamily}%
\newtheorem{definition}[theorem]{Definition}}
\theoremstyle{plain}{\theorembodyfont{\rmfamily}%
\newtheorem{problem}[theorem]{Problem}}
\usepackage{enumitem}
\setlist[enumerate]{itemsep=-1pt}
\setlist[itemize]{itemsep=-1pt}

\numberwithin{equation}{section}
\usepackage{authblk}

\newcommand{\email}[1]{\href{mailto:#1}{\nolinkurl{#1}}}
\author{Minh N. B\`ui}
\affil{North Carolina State University,
Department of Mathematics,
Raleigh, NC 27695-8205, USA\\
\email{mnbui@ncsu.edu}
}
\begin{document}

\title{\sffamily\huge\vskip -10mm
A Decomposition Method for Solving\\
Multicommodity Network Equilibrium
}

\date{~}

\maketitle

\vskip -6mm

\begin{abstract}
We consider the numerical aspect of the multicommodity network
equilibrium problem proposed by Rockafellar in 1995. Our method
relies on the flexible monotone operator splitting framework
recently proposed by Combettes and Eckstein.
\end{abstract}

\section{Problem formulation}
\label{sec:1}

Rockafellar proposed in \cite{Rock95} the important multicommodity
network equilibrium model (see \eqref{e:1} in
Problem~\ref{prob:1}) and studied some of its properties.
In the present paper, we devise a flexible numerical method for
solving this problem based on the asynchronous block-iterative
decomposition framework of \cite{MaPr18}.

The following notion of a network from \cite[Section~1A]{Rock84}
plays a central role in our problem.

\begin{definition}
\label{d:1}
A network consists of nonempty finite sets $\BN$ and $\BA$ ---
whose elements are called nodes and arcs, respectively --- and a
mapping $\vartheta\colon\BA\to\BN\times\BN\colon
j\mapsto(\vartheta_1(j),\vartheta_2(j))$ such that, for every
$j\in\BA$, $\vartheta_1(j)\neq\vartheta_2(j)$. We call
$\vartheta_1(j)$ and $\vartheta_2(j)$ the initial node and the
terminal node of arc $j$, respectively. In addition, we set
\begin{equation}
\label{e:ai+-}
(\forall i\in\BN)\quad
\begin{cases}
\BA^+(i)=\menge{j\in\BA}{
\text{node $i$ is the initial node of arc $j$}}\\
\BA^-(i)=\menge{j\in\BA}{
\text{node $i$ is the terminal node of arc $j$}}.
\end{cases}
\end{equation}
\end{definition}

Recall that, given a Euclidean space $\GG$ with scalar product
$\scal{\mute}{\mute}$, an operator $A\colon\GG\to 2^\GG$ is
maximally monotone if
\begin{equation}
\big(\forall(x,x^*)\in\GG\times\GG\big)\quad
(x,x^*)\in\gra A\quad\Leftrightarrow\quad
\big[\;\big(\forall(y,y^*)\in\gra A\big)\;\;
\scal{x-y}{x^*-y^*}\geq 0\;\big],
\end{equation}
where $\gra A=\menge{(x,x^*)\in\GG\times\GG}{x^*\in Ax}$
is the graph of $A$. (The reader is referred to \cite{Livre1} for
background and complements on monotone operator theory and convex
analysis.) The problem of interest is the following.

\begin{problem}
\label{prob:1}
Under consideration is a network $(\BN,\BA,\vartheta)$,
together with a nonempty finite set $\BCC$ of commodities
transiting on the network. Equip $\HH=\RR^\BCC$ with the
scalar product $((\xi_k)_{k\in\BCC},(\eta_k)_{k\in\BCC})\mapsto
\sum_{k\in\BCC}\xi_k\eta_k$ and let us introduce the spaces
\begin{equation}
\label{e:2250}
\begin{cases}
\XXX=
\menge{\boldsymbol{x}=(x_j)_{j\in\BA}}{
(\forall j\in\BA)\;\;x_j=(\xi_{j,k})_{k\in\BCC}\in\HH}\\
\VVV=\menge{\boldsymbol{v}^*=(v_i^*)_{i\in\BN}}{
(\forall i\in\BN)\;\;v_i^*=(\nu_{i,k}^*)_{k\in\BCC}\in\HH}.
\end{cases}
\end{equation}
An element $\boldsymbol{x}\in\XXX$ is called a flow on the
network, where $\xi_{j,k}$ is the flux of commodity $k$ on arc
$j$. The divergence of a flow $\boldsymbol{x}\in\XXX$ at node $i$
is
\begin{equation}
\label{e:dive}
\dive_i\boldsymbol{x}=
\sum_{j\in\BA^+(i)}x_j-\sum_{j\in\BA^-(i)}x_j.
\end{equation}
We refer to an element $\boldsymbol{v}^*\in\VVV$ as a potential on
the network, where $\nu_{i,k}^*$ is the potential of commodity $k$
at node $i$. Given $\boldsymbol{v}^*\in\VVV$ and $j\in\BA$, the
tension (or potential difference) across arc $j$ relative to the
potential $\boldsymbol{v}^*$ is
\begin{equation}
\label{e:tens}
\Delta_j\boldsymbol{v}^*=
v_{\vartheta_2(j)}^*-v_{\vartheta_1(j)}^*.
\end{equation}
For every $j\in\BA$, the flow-tension relation on arc $j$
is modeled by the sum
$Q_j+R_j$ of maximally monotone operators
$Q_j\colon\HH\to 2^\HH$ and $R_j\colon\HH\to 2^\HH$.
Further, for every $i\in\BN$, the
divergence-potential relation
at node $i$ is modeled by a maximally monotone operator
$S_i\colon\HH\to 2^\HH$. The task is to
\begin{equation}
\label{e:1}
\text{find a flow}\;\:
\overline{\boldsymbol{x}}\in\XXX\;\:
\text{and a potential}\;\:\overline{\boldsymbol{v}}^*\in
\VVV\;\:\text{such that}\;\:
\begin{cases}
(\forall j\in\BA)\;\;
\Delta_j\overline{\boldsymbol{v}}^*\in
Q_j\overline{x}_j+R_j\overline{x}_j\\
(\forall i\in\BN)\;\;\dive_i\overline{\boldsymbol{x}}\in
S_i^{-1}\overline{v}_i^*,
\end{cases}
\end{equation}
under the assumption that \eqref{e:1} has a solution.
\end{problem}

\begin{remark}
\label{r:8}
The pertinence of Problem~\ref{prob:1} is demonstrated in
\cite[Chapter~8]{Rock84} and \cite{Rock95}, where it is shown to
capture formulations arising in areas such as traffic assignment,
hydraulic networks, and price equilibrium.
\end{remark}

\section{A block-iterative decomposition method}
\label{sec:2}

{\noindent\bfseries Notation.}
Throughout, $\GG$ is a Euclidean space. Let $A\colon\GG\to 2^\GG$
be maximally monotone and let $x\in\GG$. Then, in terms of the
variable $p\in\GG$, the inclusion $x\in p+Ap$ has a unique
solution, which is denoted by $J_Ax$. The operator
$J_A\colon\GG\to\GG\colon x\mapsto J_Ax$ is called the resolvent
of $A$.

Our algorithm (see \eqref{e:alg} in Proposition~\ref{p:1}) is
derived from \cite[Algorithm~12]{MaPr18} and it thus inherits the
following attractive features from the framework of \cite{MaPr18}:
\begin{itemize}
\item
No additional assumption, such as Lipschitz continuity or
cocoercivity, is imposed on the underlying operators. 
\item
Algorithm~\eqref{e:alg} achieves full splitting in the sense that
the operators $(Q_j)_{j\in\BA}$, $(R_j)_{j\in\BA}$, and
$(S_i)_{i\in\BN}$ are activated independently via their
resolvents.
\item
Algorithm~\eqref{e:alg} is block-iterative, that is,
at iteration $n$, only blocks
$(Q_j)_{j\in\BA_n}$, $(R_j)_{j\in\BA_n}$, and $(S_i)_{i\in\BN_n}$
of operators need to be activated. To guarantee convergence
of the iterates, the mild sweeping condition
\eqref{e:2131} needs to be fulfilled.
\end{itemize}

We shall denote elements in $\XXX$ and $\VVV$ by bold letters,
e.g., $\boldsymbol{q}_n=(q_{j,n})_{j\in\BA}$ and
$\boldsymbol{s}_n^*=(s_{i,n}^*)_{i\in\BN}$.

\begin{proposition}
\label{p:1}
Consider the setting of Problem~\ref{prob:1}. Let $T\in\NN$,
let $(\BA_n)_{n\in\NN}$ be nonempty subsets of $\BA$, and
let $(\BN_n)_{n\in\NN}$ be nonempty subsets of $\BN$
such that $\BA_0=\BA$, $\BN_0=\BN$, and
\begin{equation}
\label{e:2131}
(\forall n\in\NN)\quad
\bigcup_{k=n}^{n+T}\BA_k=\BA
\quad\text{and}\quad
\bigcup_{k=n}^{n+T}\BN_k=\BN.
\end{equation}
Let $(\lambda_n)_{n\in\NN}$ be a sequence in $\left]0,2\right[$
such that $\inf_{n\in\NN}\lambda_n>0$ and
$\sup_{n\in\NN}\lambda_n<2$. Moreover,
for every $j\in\BA$ and every $i\in\BN$,
let $(x_{j,0},x_{j,0}^*,v_{i,0}^*)\in\HH^3$
and $(\gamma_j,\mu_j,\sigma_i)\in\RPP^3$. Iterate
\begin{equation}
\label{e:alg}
\begin{array}{l}
\text{for}\;n=0,1,\ldots\\
\left\lfloor
\begin{array}{l}
\text{for every}\;j\in\BA_n\\
\left\lfloor
\begin{array}{l}
l_{j,n}^*=x_{j,n}^*-\Delta_j\boldsymbol{v}_n^*\\
q_{j,n}=J_{\gamma_jQ_j}(x_{j,n}-\gamma_jl_{j,n}^*)\\
q_{j,n}^*=\gamma_j^{-1}(x_{j,n}-q_{j,n})-l_{j,n}^*\\
r_{j,n}=J_{\mu_jR_j}(x_{j,n}+\mu_jx_{j,n}^*)\\
r_{j,n}^*=x_{j,n}^*+\mu_j^{-1}(x_{j,n}-r_{j,n})
\end{array}
\right.
\\
\text{for every}\;j\in\BA\smallsetminus\BA_n\\
\left\lfloor
\begin{array}{l}
q_{j,n}=q_{j,n-1};\;
q_{j,n}^*=q_{j,n-1}^*;\;
r_{j,n}=r_{j,n-1};\;
r_{j,n}^*=r_{j,n-1}^*
\end{array}
\right.
\\
\text{for every}\;i\in\BN_n\\
\left\lfloor
\begin{array}{l}
l_{i,n}=
\dive_i\boldsymbol{x}_n\\
s_{i,n}=J_{\sigma_iS_i}(l_{i,n}+\sigma_iv_{i,n}^*)\\
s_{i,n}^*=v_{i,n}^*+
\sigma_i^{-1}(l_{i,n}-s_{i,n})\\
t_{i,n}=s_{i,n}-\dive_i\boldsymbol{q}_n
\end{array}
\right.
\\
\text{for every}\;i\in\BN\smallsetminus\BN_n\\
\left\lfloor
\begin{array}{l}
s_{i,n}=s_{i,n-1};\;
s_{i,n}^*=s_{i,n-1}^*\\
t_{i,n}=s_{i,n}-\dive_i\boldsymbol{q}_n
\end{array}
\right.
\\
\text{for every}\;j\in\BA\\
\left\lfloor
\begin{array}{l}
t_{j,n}^*=q_{j,n}^*+r_{j,n}^*-\Delta_j\boldsymbol{s}_n^*\\
u_{j,n}=r_{j,n}-q_{j,n}
\end{array}
\right.
\\
\tau_n=\sum_{j\in\BA}\big(\|t_{j,n}^*\|^2+
\|u_{j,n}\|^2\big)+\sum_{i\in\BN}\|t_{i,n}\|^2\\
\text{if}\;\tau_n>0\\
\left\lfloor
\begin{array}{l}
\begin{aligned}
\pi_n&=\textstyle\sum_{j\in\BA}\big(
\scal{x_{j,n}}{t_{j,n}^*}-
\scal{q_{j,n}}{q_{j,n}^*}+
\scal{u_{j,n}}{x_{j,n}^*}-
\scal{r_{j,n}}{r_{j,n}^*}\big)
\\
&\textstyle\quad\;+\sum_{i\in\BN}\big(
\scal{t_{i,n}}{v_{i,n}^*}-
\scal{s_{i,n}}{s_{i,n}^*}\big)
\end{aligned}
\\
\theta_n=\lambda_n\max\{\pi_n,0\}/\tau_n
\end{array}
\right.
\\
\text{else}\\
\left\lfloor
\begin{array}{l}
\theta_n=0
\end{array}
\right.
\\
\text{for every}\;j\in\BA\\
\left\lfloor
\begin{array}{l}
x_{j,n+1}=x_{j,n}-\theta_nt_{j,n}^*\\
x_{j,n+1}^*=x_{j,n+1}^*-\theta_nu_{j,n}
\end{array}
\right.
\\
\text{for every}\;i\in\BN\\
\left\lfloor
\begin{array}{l}
v_{i,n+1}^*=v_{i,n}^*-\theta_nt_{i,n}.
\end{array}
\right.\\[2mm]
\end{array}
\right.
\end{array}
\end{equation}
Then $((x_{j,n})_{j\in\BA},(v_{i,n}^*)_{i\in\BN})_{n\in\NN}$
converges to a solution to \eqref{e:1}.
\end{proposition}
\begin{proof}
Let us consider the multivariate monotone inclusion problem
\begin{multline}
\label{e:20}
\text{find}\;\:
\overline{\boldsymbol{x}}\in\XXX,\;\:
\overline{\boldsymbol{x}}^*\in\XXX,\;\:
\text{and}\;\:\overline{\boldsymbol{v}}^*\in\VVV
\;\:\text{such that}\;\:
\begin{cases}
(\forall j\in\BA)\;\;\Delta_j\overline{\boldsymbol{v}}^*
-\overline{x}_j^*\in Q_j\overline{x}_j\;\:\text{and}\;\:
\overline{x}_j\in R_j^{-1}\overline{x}_j^*\\
(\forall i\in\BN)\;\;
\dive_i\overline{\boldsymbol{x}}\in S_i^{-1}\overline{v}_i^*.
\end{cases}
\end{multline}
Then
\begin{align}
\label{e:1940}
&(\forall\overline{\boldsymbol{x}}\in\XXX)
(\forall\overline{\boldsymbol{v}}^*\in\VVV)\quad
(\overline{\boldsymbol{x}},
\overline{\boldsymbol{v}}^*)\;\:\text{solves \eqref{e:1}}
\nonumber\\
&\hskip 40mm
\Leftrightarrow(\exi\overline{\boldsymbol{x}}^*\in\XXX)\;\;
\begin{cases}
(\forall j\in\BA)\;\;\Delta_j\overline{\boldsymbol{v}}^*
\in Q_j\overline{x}_j+\overline{x}_j^*\;\:\text{and}\;\:
\overline{x}_j^*\in R_j\overline{x}_j\\
(\forall i\in\BN)\;\;
\dive_i\overline{\boldsymbol{x}}\in S_i^{-1}\overline{v}_i^*
\end{cases}
\nonumber\\
&\hskip 40mm
\Leftrightarrow(\exi\overline{\boldsymbol{x}}^*\in\XXX)\;\;
(\overline{\boldsymbol{x}},\overline{\boldsymbol{x}}^*,
\overline{\boldsymbol{v}}^*)\;\:\text{solves \eqref{e:20}}.
\end{align}
Therefore, since \eqref{e:1} has a solution, so does \eqref{e:20}.
Next, define
\begin{equation}
\label{e:na}
(\forall i\in\BN)(\forall j\in\BA)\quad
\varepsilon_{i,j}=
\begin{cases}
1,&\text{if node $i$ is the initial node of arc $j$};\\
{-}1,&\text{if node $i$ is the terminal node of arc $j$};\\
0,&\text{otherwise}.
\end{cases}
\end{equation}
It results from \eqref{e:dive} and \eqref{e:ai+-} that
\begin{equation}
\label{e:dive2}
(\forall\boldsymbol{x}\in\XXX)(\forall i\in\BN)\quad
\dive_i\boldsymbol{x}=\sum_{j\in\BA}\varepsilon_{i,j}x_j,
\end{equation}
and from \eqref{e:tens} that
\begin{equation}
\label{e:tens2}
(\forall\boldsymbol{v}^*\in\VVV)(\forall j\in\BA)\quad
\Delta_j\boldsymbol{v}^*=
{-}\sum_{i\in\BN}\varepsilon_{i,j}v_i^*.
\end{equation}
We now verify that \eqref{e:20} is a special case of
\cite[Problem~1]{MaPr18} with the setting
\begin{equation}
\label{e:7743}
I=\BA,\quad K=\BA\cup\BN,\quad\text{and}\quad
(\forall j\in I)(\forall k\in K)\;\;
\begin{cases}
\HH_j=\GG_k=\HH\\
A_j=Q_j\\
B_k=
\begin{cases}
R_k,&\text{if}\;\:k\in\BA;\\
S_k,&\text{if}\;\:k\in\BN
\end{cases}
\\
z_j^*=r_k=0\\
L_{k,j}=
\begin{cases}
\Id,&\text{if}\;\:k=j;\\
0,&\text{if}\;\:k\in\BA\;\:\text{and}\;\:k\neq j;\\
\varepsilon_{k,j}\Id,&\text{if}\;\:k\in\BN.
\end{cases}
\end{cases}
\end{equation}
We deduce from \eqref{e:dive2} that
\begin{align}
(\forall\boldsymbol{x}\in\XXX)
(\forall k\in K)\quad
\sum_{j\in I}L_{k,j}x_j
&=
\begin{cases}
x_k,&\text{if}\;\:k\in\BA;\\
\sum_{j\in I}\varepsilon_{k,j}x_j,&\text{if}\;\:k\in\BN
\end{cases}
\nonumber\\
&=
\begin{cases}
x_k,&\text{if}\;\:k\in\BA;\\
\dive_k\boldsymbol{x},&\text{if}\;\:k\in\BN,
\end{cases}
\end{align}
and from \eqref{e:tens2} that
\begin{equation}
(\forall\boldsymbol{x}^*\in\XXX)
(\forall\boldsymbol{v}^*\in\VVV)
(\forall j\in I)\quad
\sum_{k\in\BA}L_{k,j}^*x_k^*+
\sum_{k\in\BN}L_{k,j}^*v_k^*
=x_j^*+\sum_{k\in\BN}\varepsilon_{k,j}v_k^*
=x_j^*-\Delta_j\boldsymbol{v}^*.
\end{equation}
Hence, in the setting of \eqref{e:7743},
\eqref{e:20} is an instantiation of
\cite[Problem~1]{MaPr18} and \eqref{e:alg} is a realization
of \cite[Algorithm~12]{MaPr18}, where
$(\forall n\in\NN)$ $I_n=\BA_n$ and $K_n=\BA_n\cup\BN_n$.
Thus, upon letting
\begin{equation}
(\forall n\in\NN)\quad
\boldsymbol{x}_n=(x_{j,n})_{j\in\BA},\quad
\boldsymbol{x}_n^*=(x_{j,n}^*)_{j\in\BA},
\quad\text{and}\quad
\boldsymbol{v}_n^*=(v_{i,n}^*)_{i\in\BN},
\end{equation}
we infer from \cite[Theorem~13]{MaPr18} that
$(\boldsymbol{x}_n,\boldsymbol{x}_n^*,
\boldsymbol{v}_n^*)_{n\in\NN}$ converges to a solution
$(\overline{\boldsymbol{x}},
\overline{\boldsymbol{x}}^*,
\overline{\boldsymbol{v}}^*)$ to \eqref{e:20}.
Consequently, \eqref{e:1940} asserts that
$(\overline{\boldsymbol{x}},\overline{\boldsymbol{v}}^*)$
solves \eqref{e:1}.
\end{proof}

\newpage

\begin{remark}
\label{r:2}
Some comments are in order.
\begin{enumerate}
\item
One might be tempting to consider \eqref{e:1} as a special case
of \cite[Problem~1]{MaPr18} with the setting
\begin{equation}
\label{e:1516}
I=\BA,\quad K=\BN,\quad\text{and}\quad
(\forall j\in I)(\forall k\in K)\;\;
\begin{cases}
\HH_j=\GG_k=\HH\\
A_j=Q_j+R_j\\
B_k=S_k\\
z_j^*=r_k=0\\
L_{k,j}=\varepsilon_{k,j}\Id,
\end{cases}
\end{equation}
where $(\varepsilon_{i,j})_{i\in\BN,j\in\BA}$ are defined in
\eqref{e:na}, and then specialize \cite[Algorithm~12]{MaPr18}
to \eqref{e:1516}. However, this approach necessitates
the computation of the resolvents of the operators
$(Q_j+R_j)_{j\in\BA}$,
which cannot be expressed in terms of the resolvents
of $(Q_j)_{j\in\BA}$ and $(R_j)_{j\in\BA}$ in general
(see Examples~\ref{ex:27} and \ref{ex:271}).
\item
Algorithm~\eqref{e:alg} of Proposition~\ref{p:1} requires
to evaluate the resolvents of the operators
$(Q_j)_{j\in\BA}$, $(R_j)_{j\in\BA}$, and $(S_i)_{i\in\BN}$.
Illustrations of such calculations in some special cases of
Problem~\ref{prob:1} encountered in the literature
are provided in Examples~\ref{ex:27}, \ref{ex:271} and
\ref{ex:bpr}--\ref{ex:power}.
\item
Alternate algorithms \cite{Siop13,Siop15,Pesq15}
can also be used to solve \eqref{e:20} and, in turn, \eqref{e:1}.
Nevertheless, there are certain restrictions on the resulting
algorithms. For example, the method of \cite{Siop13} must activate
all the operators $(Q_j)_{j\in\BA}$, $(R_j)_{j\in\BA}$, and
$(S_i)_{i\in\BN}$ at every iteration, while the
frameworks of \cite{Siop15,Pesq15} do not allow for deterministic
selections of the blocks $(Q_j)_{j\in\BA_n}$, $(R_j)_{j\in\BA_n}$,
and $(S_i)_{i\in\BN_n}$. Finally, the algorithm resulted from
\cite{Siop15} involves the inversion of a linear operator acting on
$\RR^{MN}$, where $N=\card\BCC$ and
$M=2\card\BA+\card\BN$, which may not be favorable in large-scale
problems, e.g., \cite{Flor76}.
\end{enumerate}
\end{remark}

{\noindent\bfseries Notation.}
Before proceeding further,
let us recall some basic notion of convex analysis
(see \cite{Livre1} for details).
Let $\varphi\colon\GG\to\RX$ be proper, lower
semicontinuous, and convex. The subdifferential of $\varphi$ is
the maximally monotone operator $\partial\varphi\colon\GG\to
2^\GG\colon x\mapsto\menge{x^*\in\GG}{(\forall y\in\GG)\;\;
\scal{y-x}{x^*}+\varphi(x)\leq\varphi(y)}$. For every $x\in\GG$,
the unique minimizer of $\varphi+(1/2)\|\mute-x\|^2$
is denoted by $\prox_\varphi x$. Let $C$ be a nonempty closed
convex subset of $\GG$. The indicator function of $C$ is the
proper lower semicontinuous convex function
\begin{equation}
\iota_C\colon\GG\to\RPX\colon x\mapsto
\begin{cases}
0,&\text{if}\;\:x\in C;\\
\pinf,&\text{otherwise},
\end{cases}
\end{equation}
the normal cone operator of $C$ is $N_C=\partial\iota_C$,
and the projector onto $C$ is $\proj_C=\prox_{\iota_C}$.

\begin{example}[Separable multicommodity flows]
\label{ex:27}
Consider the setting of Problem~\ref{prob:1} and
suppose, in addition, that the following are satisfied:
\begin{enumerate}[label={\rm[\alph*]}]
\item
For every $j\in\BA$,
$\fc_j\colon\RR\to 2^\RR$ is maximally monotone,
$C_j$ is a nonempty closed convex subset of $\HH$, and
\begin{equation}
\label{e:3237}
Q_j\colon\HH\to 2^\HH\colon x_j=(\xi_{j,k})_{k\in\BCC}\mapsto
\Bigg(\fc_j\bigg(\sum_{k\in\BCC}\xi_{j,k}\bigg)\Bigg)_{k\in\BCC}
\quad\text{and}\quad
R_j=N_{C_j}.
\end{equation}
\item
For every $i\in\BN$,
$s_i\in\HH$ and
\begin{equation}
\label{e:3238}
S_i^{-1}\colon\HH\to 2^\HH\colon v_i^*\mapsto\{s_i\}.
\end{equation}
\end{enumerate}
Then \eqref{e:1} reduces to the separable multicommodity flow
problem; see, e.g., \cite[Section~8.3]{Bert98} and the references
listed in \cite[Section~8.9]{Bert98}.
Take $j\in\BA$, $i\in\BN$, and $\gamma\in\RPP$. We have
$J_{\gamma R_j}=\proj_{C_j}$ and $J_{\gamma S_i}=s_i$.
To compute $J_{\gamma Q_j}$, define
$L\colon\HH\to\RR\colon (\xi_k)_{k\in\BCC}\mapsto
\sum_{k\in\BCC}\xi_k$ and set $N=\card\BCC$.
Then $L^*\colon\RR\to\HH\colon\xi\mapsto(\xi)_{k\in\BCC}$
and, therefore, $L\circ L^*=N\Id$. At the same time,
by \eqref{e:3237}, $Q_j=L^*\circ\fc_j\circ L_j$.
Thus, we derive from \cite[Proposition~23.25(iii)]{Livre1} that
\begin{multline}
\big(\forall x_j=(\xi_{j,k})_{k\in\BCC}\in\HH\big)\quad
J_{\gamma Q_j}x_j
=x_j+\frac{1}{N}\big(J_{N\gamma\fc_j}(Lx_j)-Lx_j\big)_{k\in\BCC}
=(\xi_{j,k}+\eta)_{k\in\BCC},\\
\text{where}\;\:\eta=
\Bigg(J_{N\gamma\fc_j}\Bigg(\sum_{k\in\BCC}\xi_{j,k}\Bigg)-
\sum_{k\in\BCC}\xi_{j,k}\Bigg)\bigg/N.
\end{multline}
\end{example}

\begin{example}
\label{ex:271}
The separable multicommodity flow problem with arc
capacity constraints (see, e.g., \cite[Section~8.3]{Bert98})
is an instantiation of Example~\ref{ex:27}
with, for every $j\in\BA$,
$\fc_j=\partial(\phi_j+\iota_{\Omega_j})$,
where $\phi_j\colon\RR\to\RX$ is a proper lower semicontinuous
convex function and $\Omega_j$ is a nonempty closed interval
in $\RR$ such that $\Omega_j\cap\dom\phi_j\neq\emp$.
In this setting, it follows from
\cite[Example~23.3 and Proposition~24.47]{Livre1} that
\begin{equation}
(\forall j\in\BA)\big(\forall\gamma\in\RPP\big)\quad
J_{\gamma\fc_j}
=\prox_{\gamma(\phi_j+\iota_{\Omega_j})}
=\proj_{\Omega_j}\circ\prox_{\gamma\phi_j}.
\end{equation}
\end{example}

\begin{remark}
\label{r:1}
Consider the standard traffic assignment problem, that is,
the special case of Example~\ref{ex:27} where
$(\forall j\in\BA)$ $C_j=\RP^\BCC$.
\begin{enumerate}
\item
In \cite[Example~4.4]{Sico10},
this problem was solved by an application of the
forward-backward method \cite[Theorem~2.8]{Sico10},
where it is further assumed that, for
every $j\in\BA$, $\dom\fc_j=\RR$ and $\fc_j$ is Lipschitzian.
However, some common operators found in the literature of traffic
assignment \cite{Bran76} do not fulfill this requirement; their
resolvents are provided in Examples~\ref{ex:bpr}--\ref{ex:power}.
\item
The method of \cite{Fuku96}, which is an application of the
Douglas--Rachford algorithm \cite{Lion79},
requires to compute the projectors onto polyhedral sets of the form
\begin{equation}
\label{e:4419}
\Menge{(\xi_j)_{j\in\BA}\in\RP^\BA}{(\forall i\in\BN)\;\;
\sum_{j\in\BA}\varepsilon_{i,j}\xi_j=\delta_i},\quad
\end{equation}
\end{enumerate}
where $(\varepsilon_{i,j})_{i\in\BN,j\in\BA}$ are defined in
\eqref{e:na}. This results in solving a subproblem
at every iteration because there is no closed-form
expression for such projectors.
\end{remark}

\begin{example}[Bureau of Public Roads capacity operator]
\label{ex:bpr}
Let $(\alpha,\varrho,\theta,p)\in\RPP^4$ and define
\begin{equation}
\fc\colon\RR\to\RR\colon\xi\mapsto
\begin{cases}
\theta\bigg(1+
\alpha\bigg(\dfrac{\xi}{\varrho}\bigg)^{\!p}\bigg),
&\text{if}\;\:\xi\geq 0;\\
\theta,&\text{if}\;\:\xi<0.
\end{cases}
\end{equation}
In addition, let $\gamma\in\RPP$ and $\xi\in\RR$.
Then the following hold:
\begin{enumerate}
\item
Suppose that $\xi\geq\gamma \theta$. Then,
in terms of the variable $s\in\RR$, the equation
\begin{equation}
\frac{\alpha\gamma \theta}{\varrho^p}s^p+s+\gamma\theta-\xi=0
\end{equation}
has a unique solution $\overline{s}$ and
$J_{\gamma\fc}\xi=\overline{s}$.
\item
Suppose that $\xi<\gamma\theta$.
Then $J_{\gamma\fc}\xi=\xi-\gamma\theta$.
\end{enumerate}
\end{example}

\begin{example}[Logarithmic capacity operator]
\label{ex:loga}
Let $\omega\in\RPP$, let $\theta\in\RP$, and define
\begin{equation}
\label{e:loga}
\fc\colon\RR\to 2^\RR\colon\xi\mapsto
\begin{cases}
\bigg\{\theta+\ln\dfrac{\omega}{\omega-\xi}\bigg\},
&\text{if}\;\:\xi<\omega;\\
\emp,&\text{if}\;\:\xi\geq\omega.
\end{cases}
\end{equation}
Then
\begin{equation}
\big(\forall\gamma\in\RPP\big)(\forall\xi\in\RR)\quad
J_{\gamma\fc}\xi=\omega-\gamma\EuScript{W}\big(
\omega\gamma^{-1}\exp(\theta-\xi/\gamma+\omega/\gamma)\big),
\end{equation}
where $\EuScript{W}$ is the Lambert W-function, that is, the
inverse of $\left[{-}1,\pinf\right[\to
\left[{-}1/e,\pinf\right[\colon\xi\mapsto\xi\exp(\xi)$.
\end{example}

\begin{example}[Traffic Research Corporation capacity operator]
\label{ex:trc}
Let $(\alpha,\beta,\delta,\omega)\in\RPP^4$ and
define
\begin{equation}
\label{e:trc}
\fc\colon\RR\to\RR\colon\xi\mapsto
\delta+\alpha(\xi-\omega)+\sqrt{\alpha^2(\xi-\omega)^2+\beta}.
\end{equation}
Then
\begin{equation}
\big(\forall\gamma\in\RPP\big)(\forall\xi\in\RR)\quad
J_{\gamma\fc}\xi=
\frac{{-}\sqrt{\gamma^2\alpha^2(\xi-\gamma\delta-\omega)^2
+(2\gamma\alpha+1)\gamma^2\beta}+\gamma\alpha(\xi-\gamma\delta+\omega)
+\xi-\gamma\delta}{2\gamma\alpha+1}.
\end{equation}
\end{example}

\begin{example}
\label{ex:power}
Let $\alpha\in\left]1,\pinf\right[$,
let $\theta\in\RPP$, let $p\in\RPP$, and define
\begin{equation}
\label{e:power}
\fc\colon\RR\to\RR\colon\xi\mapsto\theta\alpha^{p\xi}.
\end{equation}
Then
\begin{equation}
\big(\forall\gamma\in\RPP\big)(\forall\xi\in\RR)\quad
J_{\gamma\fc}\xi=\xi-\frac{\EuScript{W}\big(
\gamma\theta\alpha^{p\xi}p\ln\alpha\big)}{p\ln\alpha}.
\end{equation}
\end{example}

\begin{ack}
This work is a part of the author's Ph.D. dissertation and it was
supported by the National Science Foundation under grant
CCF-1715671. The author thanks his Ph.D. advisor P. L. Combettes
for his guidance during this work.
\end{ack}

\end{document}